\documentclass[11pt, reqno]{amsart}
\usepackage{amsmath, amsthm, a4, latexsym, amssymb, marvosym}

\def\@rmrk#1#2{\refstepcounter
    {#1}\@ifnextchar[{\@yrmrk{#1}{#2}}{\@xrmrk{#1}{#2}}}

\makeatletter\@addtoreset{equation}{section}\makeatother

 \newfont{\bfit}{cmbxti10 scaled 2000}
 \newfont{\biggi}{cmr12 scaled 2000}

 \newcommand{\eps}{\varepsilon}

 \newcommand{\R}{\mathbb{R}}
 
 \newcommand{\N}{\mathbb{N}}

 \newcommand{\prob}{\mathbb{P}}

 \newcommand{\me}{\mathbb{E}}
 
 \renewcommand{\P}{\mathbb{P}}
 
 \newcommand{\tome}{\tilde{\omega}}

 \newcommand{\skric}{{\mathcal C}}
 \newcommand{\skrid}{{\mathcal D}}

 \newcommand{\skril}{{\mathcal L}}
 \newcommand{\skrim}{{\mathcal M}}

 \newcommand{\skrin}{{\mathcal N}}

 \newcommand{\skrix}{{\mathcal X}}

 \newcommand{\sfrac}[2]{\mbox{$\frac{#1}{#2}$}}

\def\1{{\mathchoice {1\mskip-4mu\mathrm l}      
{1\mskip-4mu\mathrm l}
{1\mskip-4.5mu\mathrm l} {1\mskip-5mu\mathrm l}}}

\newcommand{\eq}{\begin{equation}}
\newcommand{\en}{\end{equation}}

\newenvironment{Proof}
{\vskip0.1cm\noindent{\bf Proof. }{\hspace*{0.3cm}}}%
{\nopagebreak {\hspace*{\fill}\rule{2mm}{2mm}}\\ }

{\nopagebreak {\hspace*{\fill}\rule{2mm}{2mm}}\\ }

\renewcommand{\subsection}{\secdef \subsct\sbsect}
\newcommand{\subsct}[2][default]{\refstepcounter{subsection}
\vspace{0.15cm}
{\flushleft\bf \arabic{section}.\arabic{subsection}~\bf #1  }
\nopagebreak\nopagebreak}
\newcommand{\sbsect}[1]{\vspace{0.1cm}\noindent
{\bf #1}\vspace{0.1cm}}

\newtheorem{theorem}{Theorem}[section]
\newtheorem{lemma}[theorem]{Lemma}

\newtheoremstyle{thm}{1.5ex}{1.5ex}{\itshape\rmfamily}{}
{\bfseries\rmfamily}{}{2ex}{}

\newtheoremstyle{rem}{1.3ex}{1.3ex}{\rmfamily}{}
{\itshape\rmfamily}{}{1.5ex}{}
\theoremstyle{rem}

\refstepcounter{subsection}

\def\thebibliography#1{\section*{References}
  \list%
  {\arabic{enumi}.}
    {\settowidth\labelwidth{[#1]}\leftmargin\labelwidth
    \advance\leftmargin\labelsep
    \parsep0pt\itemsep0pt
    \usecounter{enumi}}
    \def\newblock{\hskip .11em plus .33em minus .07em}
    \sloppy                   
    \sfcode`\.=1000\relax}

\begin{document}
\title [ Large deviation principle for the  empirical degree measure of fitness preferential
attachment random networks] { Large deviation principle for the
empirical degree measure of preferential attachment random graphs }


\maketitle
\thispagestyle{empty}
\vspace{-0.5cm}

\centerline{\sc{By K. Doku-Amponsah, F.O. Mettle and E.N.N. Nortey}}
\renewcommand{\thefootnote}{}
\renewcommand{\thefootnote}{1}
\renewcommand{\thefootnote}{}

\vspace{0.5cm}

\centerline{\bf \small Abstract} \begin{quote}{\small    We consider
preferential  attachment  random graphs
 which  may be  obtained  as  follows: It  starts  with  a  single node. If  a  new  node
appears, it is linked by an edge to one or more existing node(s)
with a probability proportional to function of their degree. For a
class of  linear preferential attachment random  graphs we find a
 large  deviation  principle (LDP) for the empirical  degree measure. In  the course  of  the  prove  this  LDP we establish an  LDP
for the empirical degree and pair distribution see
Theorem~\ref{PA3a},  of the fitness preferential attachment model of
random graphs.}
\end{quote}
\vspace{0.5cm}

\section{Introduction}
Preferential attachment (P.A) random graph models have become
extremely popular in the last  two decades  since they were  first
studied by (Barabasi and Albert ,1999).  Example  (van der Hofstad
,2013), (Newman, 2003) and (Newman  et.  al, 2006) provide good
overviews.

The P.A  model of random graphs are  graphs in which nodes are added
sequentially and attach to exactly one randomly chosen existing node
and the chance a new node connects to an existing node is
proportional to its degree.

The model is typically generalized to allow for vertices to have $m
>1$ initial edges by collapsing $m$ vertices in the one initial edge
case into a single vertex (possibly causing loops). The most studied
feature of these objects is the distribution of the degrees of the
nodes; that is, the proportion of nodes that have degree k as the
graph grows large.  See, example (Collevecchio et. al, 2013),
(Krapivsky et. al, 2000), (Rudas et. al, 2007) for results on more
general attachment rules.

Few  large  deviation results for  P.A  model have so far been
found. In paper ( Choi et. al, 2011), P.A schemes where the
selection mechanism is possibly time-dependent are considered, and
an in infinite dimensional large deviation principle for the sample
path evolution of the empirical degree distribution is found by
Dupuis-Ellis type methods.

(Dereich and  Moerters, 2009) studied a dynamic model of random
networks, where new vertices are connected to old ones with a
probability proportional to a sub-linear function of their degree.
For this model of  random  networks, they  obtained a strong limit
law for the empirical degree distribution. Results on the temporal
evolution of the degrees of individual vertices via large and
moderate deviation principles  were also found.

(Bryc et. al, 2009)  found  the large  deviation principle and
related results  for a class of Markov chains associated to the
`leaves'in P.A model  of  random  graphs using both analytic and
Dupuis--Ellis-type path arguments.  Recently,(Doku-Amponsah et. al,
2014)  proved  a  large  deviation  upper  bound  for  fitness
preferential  attachment  random  network.

In this  paper,  we  find  a  large  deviation  principle   for the
empirical  degree distribution  of  preferential  attachment random
network  in  the linear  regime. See, Theorem~\ref{PA3}. In  the
course  of  the  proof  of Theorem~\ref{PA3},  we find  a  large
deviation  principle  for  the  empirical  degree  and  pair measure
of  the  fitness preferential attachment  random  networks,  see
Theorem~\ref{PA30}   and a  joint LDP for  the  empirical  degree
and  pair measure, and  the  sample  path empirical  degree
distribution of the fitness preferential attachment random networks,
see Theorem~\ref{PA3a}. The main technique in our proof  is
exponential change  of  measure, see example (Doku-Amponsah et. al,
2014)  and the  method  of mixtures, see (Biggins, 2004).

\section{Main Results}
\subsection{LDP for  the  preferential  attachment model of  random
graphs}

Let $f:\big\{0,1,2,...\big\}\to[0,\,\infty]$ be  a weight function.
We define a preferential attachment random graph as
follows:\\
 It  starts  with  single vertex serving as  root.
If a new vertex $n$ is introduced,  it connects to vertices
$v_n\in\{\,1,\ldots,n-1\,\}$ independently with probability
proportional to $ f(N(v_n)),$ where
  $N(m)$ is the in-degree of vertex $m.$

We  write $\skrin=\big\{0,1,2,...\big\}.$ In  this  paper, we shall
restrict ourself to functions of the form $$f(k)=\gamma
k+\beta,\,\mbox{ where $\gamma,\beta\in(0,\,\infty].$}$$

 We define empirical degree measure
measure $\skril$ on $\skrin$ by
$$\skril(k)=\frac{1}{n-1}\sum_{m=1}^{n-1}\delta_{N^{(m)}(j_m)}(k).$$

We denote by  $\skrim(\skrin)$ the space of probability measures on
$\skrin,$ equipped with the topology generated by total variation
metric
$\|\pi-\hat{\pi}\|:=\frac{1}{2}\sum_{k=0}^{\infty}\|\pi(k)-\hat{\pi}(k)\|.$

\begin{theorem}\label{PA3}
Suppose $X$ is  P.A random graph with  linear weight function
$f:\skrin\to[0,\infty],$  satisfying $\gamma\ge 1-\beta,$
$\log(1+\beta/\gamma)<\infty$
 and $$\sum_{k=0}^{\infty}\sfrac{1}{\gamma
k+\beta}=\infty.$$ Then, as $n\to\infty$, the empirical degree
measure $\skril,$ satisfies a large deviation principle in
$\skrim(\skrin)$ with good rate function
$$I(\ell)=H\Big(\ell\,\|\,\sfrac{(\gamma+\beta)}{f}\otimes\hat{\ell}\Big),$$
where $\sfrac{(\gamma+\beta)}{f}\otimes\hat{\ell}(k)
=\frac{(\gamma+\beta)}{f(k)}\hat{\ell}(k)$ and
$\hat{\ell}(k)=\1-\sum_{j=0}^{k}\ell(k).$
\end{theorem}

\subsection{Large-deviations for fitness P.A random
network}.
 To  establish Theorem~\ref{PA3} we  pass to  a  more  general random  preferential random  graph, the fitness or  coloured  preferential random  graph.
We write $\skrin=\N\cup\{0\}.$ Given a weight function
$f_{m/n}:\skrin\times\skrix\to[0,\,\infty],$ $m=1,2,3,...n$ and a
probability law $\mu$ on  finite alphabet $\skrix,$ we define
coloured (fitness) P.A random network with $n$ vertices as follows:
\begin{itemize}
\item Assign vertex $m=1$ (the root of the network) colour $X(m)$
according to $\mu:\skrix\to[0,\,1].$
\item If a new vertex $m$ is introduced, it gets colour $X(m)$ independently according $\mu,$
\item it connects to vertices  $v_m\in\{\,1,\ldots,m-1\,\}$ independently with
probability proportional to $$ f_{m/n}(N(v_m),A(m)),$$  where
$A(m)=\big(X(v_m),X(m)\big)$
 and $N(m)$ is the in-degree of vertex $m.$
 \item Repeat the  previous
 three  steps  until  we  have  $n$ vertices.
\end{itemize}

 We consider
$\big\{(N(v_m),A(m)):\,m=1,2,3,..., n\ldots\big\}$ under the joint
law of colour and tree. Denote by $X$ a typed tree and by $X(i)$
colour of vertex $i.$  We  write $\skrix^*=\skrix\times\skrix.$ In
this  paper, we shall restrict ourself to functions of the form
$$f_{t}(k,a)=\gamma(t,a)k+\beta(t,a),$$ where
$\gamma:(0,1]\times\skrix^* \to(0,\,\infty]$,
$\beta:(0,1]\times\skrix^*\to[0,\,\infty].$ We assume
\begin{equation}\label{equ.one} \gamma(t, a)+\beta(t,a):=c_t,\, \mbox{ for
all $(t,a)\in(0,1]\times\skrix.$}
\end{equation}

 Let
$N^{(m)}(i)$ be the degree of vertex $i$ at time $m$ and observe
that at time $n,$ the law of the fitness  P.A graph is given by
$$\begin{aligned}
\P_{f}^{(n)}(X)=\prod_{m=1}^{n}\mu(X(m))\times &
\prod_{m=2}^{n}\frac{f_{m/n}(N^{(m)}(j_m),
\,A(m))}{\sum_{i=1}^{m-1}f_{m/n}(N^{(m)}(i),\,A(m)).}
\end{aligned}
$$
 For every $X,$ we define empirical degree  and  pair  measure
measure $M_{X}$ on $\skrin\times\skrix^*$ by
$$M_{X}(k,\,a)=\frac{1}{n-1}\sum_{m=1}^{n-1}\delta_{(N^{(m)}(j_m),A(m))}(k,\,a).$$

We  write  $\ell_m(a)=\Big\{ j_m\in\{1,2,3,...,m-1\Big\}:
x(j_m)=a_1,\,x(m)=a_2\big\}$  and  for every $m=2,3,4,...,n-1$ we
define a probability measure on $\skrin\times\skrix^{*}$ by

$$ L_{\sfrac{m}{n}}^{X}(k, a)=\frac{1}{m-1}\sum_{j=1}^{m-1}\delta_{N^{(m)}(j)} (k)\1_{\{j\in\ell_m(A(m)\}}\otimes\delta_{A(m)}(a), $$

where
$$\begin{aligned}\label{Eqdef1}
\1_{\{j\in\ell_m(b)\}}\otimes\delta_{b}(a)=\left\{
\begin{array}{ll}  \1_{\{j\in\ell_m(b)\}}\,
& \mbox { if $b=a,$  }\\
0 & \mbox{otherwise.}
\end{array}\right.
\end{aligned}$$

and notice, $$L_{1}^{X}(k , a)=M_{X}(k,a).$$  We denote by
$\skrim(\skrix)$ the space of probability measures on $\skrix$
equipped with the weak topology and $\skrim(\skrin\times\skrix^{*})$
the space of probability measures on $\skrin\times\skrix^*,$
equipped with the topology generated by total variation
metric.$$\|\pi-\hat{\pi}\|:=\frac{1}{2}\sum_{(k,a)\in
\skrin\times\skrix^{*}}\|\pi(k,a)-\hat{\pi}(k,a)\|.$$

\pagebreak

\begin{theorem}\label{PA30}
Suppose $X$ is coloured P.A random graph with colour law
$\mu:\skrix\to(0,1]$ and linear weight functions $(f_t,$
$t\in(0,1])$ satisfying  $\inf_{t\in(0,1]}c_t\ge 1,$
\begin{equation}\label{persistentEqPA1}
\sup_{a\in\skrix^*}\int_{0}^{1}\log\Big(1+\beta(t,a)/\gamma(t,a)\Big)dt<
\infty \end{equation}
 and
$$\inf_{(t,a)\in(0,1]\times\skrix^*}\sum_{k=0}^{\infty}\sfrac{1}{\gamma(t,a)k+\beta(t,a)}=\infty\,
 .$$ Then, as $n\to\infty$, the  pair of  empirical  measures $\big(M_X, (L_{[nt]/n}^{X},t\in[0,1])\big)$ satisfies a large deviation
principle in $\skrim(\skrin\times\skrix^*)\times\{\nu\}$  with good
rate function
 $$\begin{aligned}\nonumber
\tilde{J}(\omega,\nu)=\left\{
\begin{array}{ll}H(\omega_{2,1}\,\|\,\mu))
+\sum_{a\in\skrix}\omega_{2}(a)\int_{[0,1]}H\Big(\omega(\cdot|a)\,\|\,\,\frac{c_t}{{f_t}}\otimes\nu_t(\cdot|a)\Big)dt,&
\mbox {if $\omega=\nu_1,$}\\
\infty & \mbox{otherwise,}
\end{array}\right.
\end{aligned}$$
where  $\omega_{2,1}$  is the $\skrix-$ marginal of the probability
measure $\omega_2$  and
$$\sfrac{c_t}{f_t(\cdot,\,a)}\otimes\nu_t(\cdot|\,a)(k)
=\frac{c_t}{f_t(k,\,a)}\nu_t(k\,|\,a).$$
\end{theorem}

Our  next  theorem  which   generalizes Theorem~\ref{PA3} is a
special case of Theorem~\ref{PA30}  above.

\begin{theorem}\label{PA3a} Suppose $X$ is coloured P.A random graph
with colour law $\mu:\skrix\to(0,1]$ and linear weight function
$f:\skrin\times\skrix^*\to[0,\infty]$  satisfying $c\ge 1,$
$$\sup_{a\in\skrix^*}\log\Big(1+\beta(a)/\gamma(a)\Big)< \infty $$ and
\begin{equation}\label{persistentEqPA}
\inf_{a\in\skrix^*}\sum_{k=0}^{\infty}\sfrac{1}{\gamma(a)k+\beta(a)}=\infty.
\end{equation}
Then, as $n\to\infty$, $M_X$ satisfies a large deviation principle
in $\skrim(\skrin\times\skrix^*)$  with good rate function

$$
J(\omega)=H(\omega_{2,1}\,\|\,\mu))
+\sum_{a\in\skrix}\omega_{2}(a)H\Big(\omega(\cdot|a)\,\|\,\,\frac{c}{{f}}\otimes\hat{\omega}(k
\,| \,a)\Big)$$

where  $\hat{\omega}(k \,| \,a):=\1-\sum_{j=0}^{k}\omega(k\,|\,a).$

\end{theorem}
Observe  that   $J(\omega)=0$   if  and  only  if
$\omega(k,a)=\frac{c\omega_2(a)}{f(k,\,a)}\big(\1-\sum_{j=0}^{k}\omega(k\,|\,a)\big),$
and  hence solving  recursively  for   $\omega(\cdot\,|\,a)$  we get
\begin{equation}\label{Rudas}
\omega(k\,|a)=\pi_f(k\,|a):=\frac{c}{c+f(k,a)}\prod_{i=0}^{k-1}\frac{f(i,a)}{c+f(i,a)}.
\end{equation}

Here we  remark  that  conditions \eqref{equ.one} and
\eqref{persistentEqPA}  are necessary  for $\pi_f(\cdot\,|a)$ to  be
a probability  measure on $\skrin$.  See (Dereich and  Morters,
2009, p.~13). Note, if  $f(k,a)=w(k)$  then \eqref{Rudas} concise
with  the asymptotic degree distribution  of random  trees and
general branching processes found in (Rudas et. al, 2008).

\section{Proof  of  Results}
\subsection{Dynamics of  the path empirical degree distribution.}
Denote  by $\skrid([0,1],\R)$ the  space  of  right continuous left
limited(cadlag) paths  from $[0,1]$  to  $\R.$   We  define  the
sample  path space
$$\begin{aligned}
&\skrid_{\skrim}:=D([0,1]:\skrim(\skrin\times\skrix))\\
&=\Big\{\mbox{the set of  all $\nu:[0,1]
\mapsto\skrim(\skrin\times\skrix)$ such that
$\nu(k,a)\in\skrid([0,1],\R)$ for  all  $k\ge 0, a\in\skrix$  and
$\langle \nu\rangle=1$}\Big\}
\end{aligned}
$$ and  endow  it with the  topology of  uniform convergence
associated with  the norm
 $$\|\nu-\hat{\nu}\|:=\sup_{t\in[0,1]}\|\nu_t-\hat{\nu}_t\|.$$

For  any $\nu\in\skrid_{\skrim}$ we  write
$\nu_t(k\,|a):=\sfrac{\nu_t(k,\,a)}{\sum_{k=0}^{\infty}\nu_t(k,\,a)},$
for  all $t\in[0,1]$  and $(k,a)\in\skrin\times\skrix.$  Write
$\dot{\nu}_{t}:=\frac{d\nu_t}{dt}$ for  the  time  derivative  of
 the  measure $\nu_t$ and we associate with each path $\nu\in\skrid_{\skrim}$  the relaxed
measure on $[0,1]\times(\skrin\times\skrix)$
 $$\bar{\nu}(dk,dt|a)=\nu_{t}(dk|a)dt.$$

We call $\nu\in\skrid_{\skrim}$ absolutely continuous if for  each
$k\in\N$, there  exists $\dot{\nu}(k|a)$ such that
$$\nu_1(k|a)-\nu_0(k|a)=\int_{0}^{1}\dot{\nu}_s(k|a)ds.$$
For  each absolutely continuous path $\nu$ ,  we define
$\nu^{\nu}(\cdot|a),$ \,$\bar{\nu}(\cdot,\cdot|a)$- almost
everywhere by

$$\nu_t^{\nu}(k|a):=-\sum_{i=0}^{k}\dot{\nu}_{t}(i|a).$$  By $\nu^{\nu}\ll\nu $
we mean $\nu$  is  absolutely  continuous. We  write
$$\skrid_{\skrim_n(\skrin\times\skrix)}:=\Big\{\nu\in\skrid_{\skrim(\skrin\times\skrix)}:\,
([nt]-1)\nu_{[nt]/n}\in\N,\,\forall t\in[0,1)\Big\}.$$

Note  that  the  measure  $L_{{[nt]}/{n}}^{X},$ for  $t\in[0,1)$ is
deterministic  and  its  distribution  is  degenerate  at  some
$\nu_{{[nt]}/{n}},$  for $t\in[0,1)$  converging  to $\nu_t,$
$t\in[0,1).$

\subsection{Exponential Change-of- Measure}
Throughout  the  remaining  part  of  this paper,
 we assume  the sample path degree distribution  $\nu$ satisfies  $\nu_t(k|a)= \nu_t^{\nu}(k|a),$ for all $t\in
 [0,\,1]$.

Let  $\tilde{g}: \N\times\skrix\to \R$, and write $\displaystyle
\lim_{n\to\infty}L_{\sfrac{[nt]}{n}}:=\nu_t\in\skrid_{\skrim},$ we
define the function $U_{\tilde{g}}:[0,\,1]\times \skrix\to \R$ by
$$ U_{\tilde{g}}^{(n)}\otimes\nu(a)=\log\Big\langle
\sfrac{e^{\tilde{g}_{[nt]/n}(\cdot,\,a)}}{f_{[nt]/n}(\cdot,\,a)},\,\nu_{\sfrac{[nt]}{n}}(\cdot|a)\Big\rangle,$$
and  note  that
$$\lim_{n\to\infty}U_{\tilde{g}_t}^{(n)}\otimes{\nu}(a)
=\log\Big\langle
\sfrac{e^{\tilde{g}_t(\cdot,\,a)}}{f_t(\cdot,\,a)},\,\nu_{t}(\cdot|a)\Big\rangle=:U_{\tilde{g}_t}\otimes{\nu}(a,t).$$
We use $\tilde{g}$ to define a new fitness P.A random  graph with
$n$ vertices as follows:
\begin{itemize}
\item At  time  $m =$ assign the root $m$ of the network fit $X(m)$  according  to  the  law
$\tilde{\mu}$ given  by
$$\tilde{\mu}(a_1)=e^{\tilde{h}(a_1)-U(\tilde{h})}\mu(a_1).$$
\item For  any  other time $m $ new node $m$  which appear  gets fit $X(m)$ according  to  the  fit  law $\tilde{\mu}.$ It connects  to node $v_m,$ independently with probability proportional to 

$$\tilde{f}_{m/n}(N^{(m)}(v_m),A(m))=\frac{c_{m/n}}{f_{m/n}(N^{(m)}(v_m),A(m))}e^{\tilde{g}_{m/n}(N^{(m)}(v_m),A(m))}.$$
\item Repeat the  previous
 three  steps  until  we  have  $n$ vertices. 



\end{itemize}

We denote by $\P_{\tilde{f},n}$ the law of the new fitness P.A graph
and observe that it is absolute continuous with respect to
$\P_{{f},n},$ as for fitness graph $X$ we have that

\begin{align}
\sfrac{d\P_{\tilde{f},n}}{d\P_{{f},n}}(X)&=\prod_{m=1}^{n}\sfrac{\tilde{\mu}(X(m)}{\mu(X(m)}
\times\sfrac{\prod_{m=1}^{n-1}\tilde{f}_{m/n}(N^{(m)}(j_m),\,A(m))}{\prod_{m=2}^{n-1}\sum_{i=1}^{m-1}\tilde{f}_{m/n}(N^{(m)}(i),\,A(m))}
\times\sfrac{\prod_{m=2}^{n-1}\sum_{i=1}^{m-1}{f}_{m/n}(N^{(m)}(i),\,A(m))}{\prod_{m=1}^{n-1}{f}_{m/n}(N^{(m)}(j_m),\,A(m))}\\
&=e^{(n-1)\Big\langle
\tilde{h}-U(\tilde{h}),\,M_X\Big\rangle+(n-1)\Big\langle\tilde{g}_{\cdot/n}-2\log
f_{\cdot/n}+\log c,\,M_{X}\otimes id \Big\rangle-(n-1)\Big\langle
U_{\tilde{g}_{\cdot/n}}\otimes L,\,M_{X}\otimes id
\Big\rangle},\label{EqPA4}
\end{align}\\

where  $id$ is the  identity function from $[0,1]$  to $[0,1].$ The
following Lemma will be used to establish the upper bound in a
variational formulation.

\begin{lemma}\label{PA5}
For every $\theta>0$ there exits a
 compact set $K_{\theta}\subset\skrim(\skrix^*)$ such that
\begin{equation}\label{EqPA6}
\limsup_{n\to\infty}\sfrac{1}{n}\log\prob_{f,n}\Big\{ M_X\not\in K\,
\big |(L_{[nt]/n}=\nu_{[nt]/n,\,\forall t\in(0,1]})\Big\} \leq -
\theta.
\end{equation}
\end{lemma}
\begin{proof}
Let  $1\ge\delta>0,$  and  $l\in N.$  We  choose  $k(l,\delta)\in
\N$ large  enough such  that,  for  large $n,$  we have
                        $$\sum_{i=1}^{[nt]-1}e^{l^2\1_{\{N^{([nt])}(i)>k(l,\delta)\}}}\sfrac{f_{[nt]/n}(N^{([nt])}(i),a)}{
                        c([nt]-1)}\le 2e^{\delta},\,\mbox{ for all $a\in \skrix$ and  for all  $t.$}$$
Now using  Chebyschev's  inequality we  have
$$
\begin{aligned}
\P_{f,n}\Big\{M_{X}(N^{([nt])}&>k(l,\delta))\ge l^{-1},\,
L_{[nt]/n}=\nu_{[nt]/n,\,\forall t\in(0,1]}\Big\}\\
&\le
e^{-nl}\me\Big\{e^{\sum_{m=1}^{n-1}l^2\1_{\{N^{(m)}(j_m)>k(l,\delta)\}}},\,L_{\sfrac{m}{n}}=\nu_{\sfrac{m}{n}},\, m=2,3,4,...,n-1\Big\}\\
&=
e^{-nl}\prod_{m=2}^{n}\me\Big\{e^{l^2\1_{\{N^{(m)}(j_m)>k(l,\delta)\}}},\,L_{\sfrac{m}{n}}=\nu_{\sfrac{m}{n}}\,\Big\}\\
&\le e^{-nl}\Big [\sup_{a\in\skrix} \sup_{t\ge
0}\Big(\sum_{i=1}^{[nt]-1}e^{l^2}\1_{\{N^{([nt])}(i)>k(l,\delta)\}}\sfrac{f_{[nt]/n}(N^{([nt])}(i),a)}{([nt]-1)
                        \Big\langle f_{[nt]/n},\,\nu_{\sfrac{[nt]}{n}}(\cdot|a)\Big\rangle}\Big)\Big]^{n}\\
&= e^{-nl}\Big [\sup_{a\in\skrix} \sup_{t\ge
0}(\sum_{i=1}^{[nt]-1}e^{l^2}\1_{\{N^{([nt])}(i)>k(l,\delta)\}}\sfrac{f_{[nt]/n}(N^{([nt])}(i),a)}{
                        c([nt]-1)})\Big]^{n}\\
 &\le  e^{-nl}\times(2e^{\delta})^{n}\\
 &=e^{n(l-\delta-\log2)}
\end{aligned}
$$

Now given  $\theta$  we choose  $M>\theta+\delta+\log2$  and  define
the set $$\Gamma_{\delta, \theta}:=\big\{\nu:\nu(N>k(l,\delta))<
l^{-1}, l\ge M\big\}$$

As  $\big\{N\le k(l,\delta)\big\}$  is  pre-compact,
$\Gamma_{\delta}$   is  compact  in  the weak topology by  prokohov
criterion. Moreover

$$\P_{f,n}\Big\{M_{X}\not\in K_{\theta}\,\big| (L_{[nt]/n}=\nu_{[nt]/n,\,\forall t\in(0,1]})\Big\}\le
\sfrac{1}{1-e^{-1}}\sfrac{e^{-\theta}}{\P\Big\{L_{[nt]/n}=\nu_{[nt]/n,\,\forall
t\in(0,1]}\Big\}}=\sfrac{1}{1-e^{-1}}e^{-\theta}.$$

Now  letting  $K_{\theta}$   be  the  closure of
$\cap_{1\ge\delta>0}\Gamma_{\delta,\theta}$ and  taking  limit as
$n$ approaches $\infty$  we  have \eqref{EqPA6}  which  ends the
proof the  Lemma.

\end{proof}

\subsection{Proof  of  Theorem~\ref{PA30}.}
We derive the upper bound in a variational formulation. To do this,
we  denote by $\skric_1$  the  space  of  all  functions  on
$\skrix$ and  by $\skric_2$  the  space  of  all  bounded continuous
functions on $\skrin\times\skrix^*.$ We define on  the space of
probability measures $\skrim(\skrin\times\skrix)$ the function
$\hat{K}$ given by

\begin{equation}\begin{aligned}\label{ratePA7}
\hat{K}_{\nu}(\omega)=\int_{[0,1]}\sup_{\tilde{g}\in
\skric_2,\tilde{h}\in\skric_1}\Big\{\int
(\tilde{h}-U(\tilde{h}))\omega_{2,1}(da_1)&+\int
\tilde{g}_t(k,a)\omega(dk,da)-2\int\log\tilde{f}_t(k,a)\omega(dk,da)\\
&+\log c_t-\int
U_{\tilde{g}_t}\otimes{\nu}(a,t)\omega_2(da)\Big\}dt.
\end{aligned}
\end{equation}

\begin{lemma}\label{PA8}
For every close set $ F\subset\skrim(\skrin\times\skrix)$ we have
\begin{equation}\label{EqPA9}
\limsup_{n\to\infty}\sfrac{1}{n}\log\prob_{f,n}\Big\{ M_X\in F \Big
|(L_{[nt]/n}=\nu_{[nt]/n,\,\forall t\in(0,1]})\Big\} \leq
-\inf_{\omega\in F }\hat{K}_{\nu}(\omega)
\end{equation}
\end{lemma}
\begin{Proof}
  We let $\tilde{h}\in\skric_1$,  $\tilde{g}\in\skric_2$ and use  the Jensen's inequality to obtain
$$\begin{aligned}
&e^{(\sup_{a_1}\tilde{h}(a)-\inf_{a_1}\tilde{h}(a_1))}\le\int
e^{\tilde{h}(X(n))-U(\tilde{h})}d\tilde{\P}_{f,n}\\
&=\me\Big\{e^{(n-1)\Big[\Big\langle
\tilde{h}-U(\tilde{h}),\,M_X\Big\rangle+\Big\langle\tilde{g}_{[nt]/n}-2\log
f_{[nt]/n}+\log c_t,\,M_{X}\otimes id\Big\rangle-\Big\langle
U_{\tilde{g}_{[nt]/n}}\otimes L,\,M_{X}\otimes id
\Big\rangle\Big]},(L_{[nt]/n}=\nu_{[nt]/n,\,\forall
t\in(0,1]})\Big\}.
\end{aligned}$$
This yields the inequality
\begin{equation}\label{EqPA10}
\limsup_{n\to\infty}\sfrac{1}{n}\log\me\Big\{e^{(n-1)\Big[\Big\langle
\tilde{h}-U(\tilde{h}),\,M_X\Big\rangle+\Big\langle\tilde{g}_{[nt]/n}-2\log
f_{[nt]/n}+\log c_t,\,M_{X}\otimes id\Big\rangle-\Big\langle
U_{\tilde{g}_{[nt]/n}}\otimes L,\,M_{X}\otimes id
\Big\rangle\Big]}\Big |(L_{[nt]/n}=\nu_{[nt]/n,\,\forall
t\in(0,1]})\Big\}=0.
\end{equation}
Given $\eps>0,$ define $\hat{K}_{\eps,\nu}$ by
$\hat{K}_{\nu,\eps}(\omega)=\min\big\{\hat{K}_{\nu}(\omega),{\eps}^{-1}\big\}-\eps.$
For $\omega\in F$ we fix $\tilde{h}\in\skric_1$ and
$\tilde{g}\in\skric_2$ such that
$$\langle
\tilde{h}-U(\tilde{h}),\,\omega_{2,1}\rangle+\langle\tilde{g}_t-2\log
f_t+\log c_t,\,\omega\otimes id \rangle-\langle
 U_{\tilde{g}_t}^{\nu},\,\omega\otimes id\rangle\ge\hat{K}_{\nu,\eps}(\omega).$$

Now, because the function $\tilde{g}_t$ is bounded, we can find open
neighbourhood $B_{\omega}$  of $\omega$, such that
\begin{equation}\label{EqPA11}
\inf_{\tilde{\omega}\in B_{\omega}}\Big\{ \langle
\tilde{h}-U(\tilde{h}),\,\omega_{2,1}\rangle+\langle\tilde{g}_t-2\log
f_t+\log c_t,\,\omega\otimes id \rangle-\langle
 U_{\tilde{g}_t}^{\nu},\,\omega\otimes id\rangle\,\Big\}
\ge\hat{K}_{\nu,\eps}(\omega)-\eps.
\end{equation}
Take  $\delta=\eps,$ apply the Chebyshev's inequality to
\eqref{EqPA11} and use \eqref{EqPA10}  to  get
\begin{equation}\label{EqPA12}
\begin{aligned}
&\limsup_{n\to\infty}\sfrac{1}{n}\log\prob_{f,n}\Big\{M_X\in B_{\omega}\Big |(L_{[nt]/n}=\nu_{[nt]/n,\,\forall t\in(0,1]})\Big\}\\
&\leq\limsup\sfrac{1}{n}\log\me\Big\{e^{(n-1)\Big[\Big\langle
\tilde{h}-U(\tilde{h}),\,M_X\Big\rangle+\Big\langle\tilde{g}_{\cdot/n}-2\log
f_{\cdot/n}+\log c_t,\,M_{X}\otimes id \Big\rangle-\Big\langle
U_{\tilde{g}_{\cdot/n}}\otimes
L,\,M_{X}\otimes id \Big\rangle\Big]}\Big |(L_{[nt]/n}=\nu_{[nt]/n,\,\forall t\in(0,1]})\Big\}\\
&\qquad\qquad-\hat{K}_{\nu,\eps}(\omega)+\eps\\
&\leq-\hat{K}_{\nu,\eps}(\omega)+2\eps
\end{aligned}
\end{equation}
Using Lemma~\ref{PA5} with $\theta=\eps^{-1} $ we may choose the
compact set $G_\eps$ such that
$$
\limsup_{n\to\infty}\sfrac{1}{n}\log\prob_{f,n}\Big\{ M_{X}\not\in
G_\eps\Big |(L_{[nt]/n}=\nu_{[nt]/n,\,\forall t\in(0,1]})\Big\} \leq
- \eps^{-1}. $$ Now, the set $F\cap G_{\eps}$ is compact and
therefore we may be covered by finitely many sets
$B_{\omega_1},\,\ldots,\,B_{\omega_r}$, with $\omega_i\in F$ , for
$i=1,\,\ldots,\,r.$ Hence, we have that

$$\begin{aligned}\prob_{f,n}\Big\{
M_X\in F \Big |L=\nu\Big\}&\leq \sum_{i=1}^{r}\prob\Big\{ M_X\in
B_{\omega_i}\Big |(L_{[nt]/n}=\nu_{[nt]/n,\,\forall
t\in(0,1]})\Big\}\\
&\qquad\qquad+\prob\Big\{ M_X\not\in G_\eps \Big
|(L_{[nt]/n}=\nu_{[nt]/n,\,\forall t\in(0,1]})\Big\}.\end{aligned}
$$

Next we use \eqref{EqPA12} we obtain for small enough $\eps>0,$
$$\begin{aligned}
\limsup_{n\to\infty}&\sfrac{1}{n}\log\prob_{f,n}\Big\{ M_X\in F\Big
|(L_{[nt]/n}=\nu_{[nt]/n,\,\forall
t\in(0,1]})\Big\}\\
&\leq
\max_{i=1}^{r}\limsup_{n\to\infty}\sfrac{1}{n}\log\prob_{f,n}\Big\{
M_X\in B_{\omega_i}\Big |(L_{[nt]/n}=\nu_{[nt]/n,\,\forall
t\in(0,1]})\Big\}-\eps^{-1}\le-\hat{K}_{\nu,\eps}(\omega)+2\eps
\end{aligned}$$

Taking $\eps\downarrow 0$ we get the desire statement.
\end{Proof}

We show that the function $\hat{K}_{\nu}(\omega)$ in Lemma~\ref{PA8}
may be replaced by the good rate function
$$K_{\nu}(\omega)=H\Big(\omega_{2,1}\,\|\,\mu\Big)
+\sum_{a\in\skrix}\omega_{2}(a)\int_{[0,1]}H\Big(\omega(\cdot|a)\,\|\,\frac{c_t}{f_t(\cdot,\,a)}\otimes\nu_t(\cdot|a)\Big)dt.$$

\begin{lemma}\label{PA13}
 For every
 $\nu\in\skrid_{\skrim}$ we have
 that $ \hat{K}_{\nu}(\omega)\ge K_{\nu}(\omega).$ Moveover, the
  function $K_{\nu}$ is good rate  function and lower semi-continuous on
  $\skrim(\skrin\times\skrix).$
 \end{lemma}
 \begin{Proof}Suppose $\nu_1=\omega$.Then, using the Jensen's inequality,
 by  our assumption  \eqref{equ.one} and  the variational characterization of entropy we  have
 $$\begin{aligned}
H\Big(\omega_{2,1}\,\|\,\mu\Big)
=\sup_{\tilde{h}}\Big\{\int\tilde{h}(a_1)\omega_{2,1}(da_1)-\log\int
e^{\tilde{h}(a_1)}\mu(da_{1})\Big\}
\end{aligned}$$

$$\begin{aligned}
&\sum_{a\in\skrix}\omega_{2}(a)\int_{[0,1]}H\Big(\omega(\cdot|a)\,\|\,\frac{c_t}{f_t(\cdot,,a)}\otimes\nu_t(\cdot|a)\Big)dt\\
&=\int_{[0,1]}\sup_{\tilde{g}_t}\Big\{\int
\tilde{g}_t(k,a)\omega(dk,da)-\log\int\int c_t
\sfrac{e^{\tilde{g}(k,\,a)}}{f_t(k,\,a)}\omega_2(da)\nu_{t}(dk|a)
\, \Big\}dt\\
&\le\int_{[0,1]}\sup_{\tilde{g}}\Big\{\int
\tilde{g}_t(k,a)\omega(dk,da)+\log c_t-2\log c_t-\int\log\Big(\int
\sfrac{e^{\tilde{g}_t(k,\,a)}}{f_t(k,\,a)}\nu_{t}(dk|a)\Big)\omega_2(da)
\, \Big\}dt\\
&=\int_{[0,1]}\sup_{\tilde{g}_t}\Big\{\int
\tilde{g}(k,a)\omega(dk,da)+\log c_t-2\log \int
f_t(k,a)\omega(dk,da)-\int\log\Big(\int
\sfrac{e^{\tilde{g}_t(k,\,a)}}{f_t(k,\,a)}\nu_{t}(dk|a)\Big)\omega_2(da)
\, \Big\}dt\\
&\le\int_{[0,1]}\sup_{\tilde{g}}\Big\{\int
\tilde{g}_t(k,a)\omega(dk,da)+\log c_t-2\int\log
f_t(k,a)\omega(dk,da)-\int\log\Big\langle
\sfrac{e^{\tilde{g}_t(\cdot|a)}}{f_t(\cdot|a)},\,\nu_{t}(\cdot|a)\Big\rangle\omega_2(da)
\, \Big\}dt\\
&=\int_{[0,1]}\sup_{\tilde{g}_t}\Big\{\int
\tilde{g}_t(k,a)\omega(dk,da)+\log c_t-2\int\log
f_t(k,a)\omega(dk,da)-\int
U_{\tilde{g}_t}(a)\omega_2(da)\Big\}dt\\
&=\hat{K}_{\nu}(\omega)
\end{aligned}$$

Recall  the definition of $K_{\nu}$   above and  notice, mapping
$\omega\to K_{\nu}(\omega)$ is  continuous function. Moreover, for
all $\alpha<\infty$, the level sets $\{K_{\nu}\le \alpha\}$ are
contained in the bounded set
$$\Big\{\omega\in\skrim(\skrin\times\skrix)\colon
\,\sum_{a\in\skrix}\omega_{2}(a)\int_{[0,1]}H\Big(\omega(\cdot|a)\,\|\,\frac{c_t}{f(\cdot,\,a)}
\otimes\nu_t^{\nu}(\cdot|a)\Big)dt\le\alpha\Big\}$$ and are
therefore compact. Consequently, $K_{\nu}$ is a good rate function.

\end{Proof}

 \subsection{Lower bound}. We establish the lower bound by using the upper
bound. To begin. we let $O$ be open subset of
$\skrim(\skrin\times\skrix)$.

 \begin{lemma}\label{PA14}
 \begin{equation}\label{EqPA15}
\liminf_{n\to\infty}\sfrac{1}{n}\log\prob_{f,n}\Big\{ M_X\in O \Big
|(L_{[nt]/n}=\nu_{[nt]/n,\,\forall t\in(0,1]})\Big\} \geq
-\inf_{\omega\in O }\hat{K}_{\nu}(\omega)
\end{equation}
\end{lemma}
\begin{Proof}Suppose $ \omega=\nu_1.$ We define
 the function $\tilde{g}_{t,\omega}:\skrix\to\R$ by
$$\begin{aligned}
\tilde{g}_{t,\omega}(k,a)=\left\{\begin{array}{ll}\log\frac{{f}_t(k,a)\omega(k|a)}{c_t\nu_{t}(k|a)}
& \mbox {if $\nu_{t}(k|a)>0,$ }\\
0 & \mbox{otherwise,}
\end{array}\right.
\end{aligned}$$

Let $B_{\omega}$  be open neighbourhood of $\omega$ such that for
all $\tilde{\omega},\in B_{\omega}$ we have that
$$\langle\tilde{g}_{t,\omega}-2\log{f}_t,\,\tilde{\omega}\rangle-\langle U_{\tilde{g}_{t,\omega}}\otimes{\nu},\,\tilde{\omega}\otimes
dt\rangle\ge\langle\tilde{g}_{t,\omega}-2\log{f}_t,,\,\omega\rangle-\langle
U_{\tilde{g}_{t,\omega}}\otimes{\nu},\,\omega\otimes
dt\rangle-2\eps.$$

We use ${\tilde\P}_{f,n}$ the law of the coloured preferential
attachment graph obtained by transforming ${\P}_{f,n}$ using
$\tilde{g}_{t,\omega}.$ We observe that colour law in the
transformed measure is $\omega_{2,1}$ and the linear weight function
is
$$\tilde{f}_t(k,a)=\sfrac{\omega(k|a)}{\nu_t(k|a)},$$ where
$$\tilde{\gamma}(t,a)=\frac{|\sum_{k=0}^{\infty}k\omega(k|a)-1|}{\sum_{k=0}^{\infty}k^2\nu_1(k|a)-1}$$
$$\tilde{\beta}(t,a)=\frac{(\sum_{k=0}^{\infty}k^2\nu_t(k|a)-1)-|\sum_{k=0}^{\infty}k\omega(k|a)-1|}
{\sum_{k=0}^{\infty}k^2\nu_t(k|a)-1}$$  and  that  therefore
$\tilde{\gamma}(t,a)+\tilde{\beta}(t,a)=1.$  We use \eqref{EqPA4} to
obtain
\begin{equation}
\begin{aligned}
&\prob_{f,n}\Big\{ M_X\in O ,\,(L_{[nt]/n}=\nu_{[nt]/n,\,\forall
t\in(0,1]})\Big\}\\
&\ge\tilde{\me}\Big\{\sfrac{{\tilde\P}_{f,n}}{{\P}_{f,n}^{\nu}}(X)\1_{\big\{{M_{X}\in
B_{\omega}\big\}}},\,(L_{[nt]/n}=\nu_{[nt]/n,\,\forall
t\in(0,1]})\Big\}\\
&=\tilde{\me}\Big\{\exp\Big(-(n-1)\langle\tilde{g}_{t,\omega}+\log
c_t- U_{\tilde{g}_{t,\omega}},\,M_{X}\otimes dt\rangle -(n-1)\langle
\log \sfrac{1}{f_t^2},\,M_{X}\otimes dt\rangle\Big)\times
\1_{\big\{{M_{X}\in
B_{\omega}\big\}}}\Big\}\\
&\ge\exp\Big(-(n-1)\langle\tilde{g}_{t,\omega}+\log c_t-
U_{\tilde{g}_{t,\omega}},\,\omega\otimes
dt\rangle+\eps\Big)\times\tilde{\P}_{f,n}\Big\{M_{X}\in
B_{\omega},\,(L_{[nt]/n}=\nu_{[nt]/n,\,\forall t\in(0,1]})\Big\}\\
&\ge\exp\Big(-(n-1)(\langle\tilde{g}_{t,\omega},\,\omega\otimes
dt\rangle-2\langle \log (\sfrac{c_t}{f_t}),\,M_{X}\otimes
dt\rangle)+\eps)\Big)\times\tilde{\P}_{f,n}\Big\{M_{X}\in
B_{\omega},\,(L_{[nt]/n}=\nu_{[nt]/n,\,\forall t\in(0,1]})\Big\}\\
&\ge\exp\Big(-(n-1)(\langle\tilde{g}_{t,\omega},\,\omega\otimes
dt\rangle +\eps)-2\int_{0}^{1}\log (
1+\beta_t/\gamma_t)dt\Big)\times\tilde{\P}_{f,n}\Big\{M_{X}\in
B_{\omega},\,(L_{[nt]/n}=\nu_{[nt]/n,\,\forall t\in(0,1]})\Big\},
\end{aligned}
\end{equation}
where  we  have  used  $c_t>1$  in  the  last  inequality.

Therefore we have that
\begin{equation}\label{EqPA16}
\begin{aligned}
\liminf_{n\to\infty}\sfrac{1}{n}\log\prob_{f,n}\Big\{ M_X\in O
\Big|&\,(L_{[nt]/n}=\nu_{[nt]/n,\,\forall t\in(0,1]})\Big\}\ge
-\langle\tilde{g}_{t,\omega},\,\omega\otimes dt\rangle+3\eps\\
&+\liminf_{n\to\infty}\sfrac{1}{n}\log\tilde\tilde{\prob}_{f,n}\Big\{
M_X\in O \Big|\,(L_{[nt]/n}=\nu_{[nt]/n,\,\forall t\in(0,1]})\Big\},
\end{aligned}
\end{equation}
where  we  have  used  \eqref{persistentEqPA1}  in  the  last
inequality.

We complete the proof of the lower bound by showing that the last
term in \eqref{EqPA16} above  vanishes. We shall use the upper bound
with the measure $\P_{f,n}$ replaced by $\tilde{\P}_{f,n}.$ Thus, by
Lemma~\ref{PA8} we have that
$$\limsup_{n\to\infty}\sfrac{1}{n}\log\tilde{\P}_{f,n}\Big\{M_{X}\in
(B_{\omega})^{c}\Big\}\leq
-\inf_{\tilde{\omega}\in(B_{\omega})^{c}}\tilde{K}_{\nu}(\tilde{\omega}),$$
$$\begin{aligned}\nonumber
\tilde{K}_{\nu}(\tilde{\omega})=\left\{
\begin{array}{ll}H(\tilde{\omega}_{2,1}\,\|\,\nu_{1,2}))
+\sum_{a\in\skrix}\tilde{\omega}_{2}(a)\int_{[0,1]}H(\tilde{\omega}(\cdot|a)\,\|\,\,\frac{1}{\tilde{f}_t}\otimes\nu_t(\cdot|a))dt&
\mbox {if $\tilde{\omega}=\nu_1,$}\\
\infty & \mbox{otherwise,}
\end{array}\right.
\end{aligned}$$
 where $A^c$ denotes complement of the set $A.$ It therefore suffice to show
 that the infimum above is positive. Suppose for contradiction
 that there exits sequence $\tilde{\omega}_n\in (B_{\omega})^{c}$ with $\tilde{K}_{\nu}(\tilde{\omega}_n)\downarrow
 0.$ Then, because the mapping
 ${\tome}\mapsto\tilde{K}_{\nu}({\tome})$
 is lower semi-continuous, we can construct a limit point $\tome\in(B_{\omega})^{c}$ with $\tilde{K}_{\nu}(\tome)=0.$ This implies that
$\tome_{2}={\nu}_{1}=\omega_{2}$ and
$\sum_{a\in\skrix}\tilde{\omega}_{2}(a)\int_{(0,1]}H(\tome(\cdot|a)\,\|\,\frac{1}{\tilde{f}_t}\otimes\nu_t(\cdot|a))dt=0.$
Hence ${\tome(k|a)}\omega(k|a) =\nu_t(k|a)\nu_t(k|a),$ for all
$k\in\skrin,$ and $t\in(0,1)$ which yields $\tilde{\omega}=\omega.$
This contradicts $\tome\in(B_{\omega})^{c}.$
\end{Proof}

 \subsection{Proof  of  Theorem \ref{PA3} By  Mixing } To  use  the
technique of  mixing  LDP  results developed in (Biggins, 2004), we
check  the  main criteria needed for the validity of (See, Biggins,
2004, Theorem~5(a))  in  the  following   Lemma. We write
$\Theta_n:=\skrid_{\skrim_{n}(\skrin\times\skrix)},$
$\Theta:=\skrid_{\skrim(\skrin\times\skrix)},$  and  define
 $$P_{f,n}(\nu_1):=\P\Big[M_X=\nu_1\,\big|\,L_{\sfrac{[nt]}{n}}^X(\cdot,a)=\nu_{\sfrac{[nt]}{n}}(\cdot,a),\,
 t\in[0,1) \mbox{ and } a\in\skrix \Big]$$
 $$P_{n}\Big(\nu_{\sfrac{[nt]}{n}},\,t\in[0,1)\Big):=\P\Big\{L_{\sfrac{[nt]}{n}}^X=\nu_{\sfrac{[nt]}{n}}\, \Big\}$$

Then,  the  joint  distribution  of $M_X$  and  $L^X$ is obtained by
the  mixture  of $P_{f,n}$ and $P_{n}$ as follows:
$$ d\tilde{P}_{f,n}(\nu,\,\nu_1):=
dP_{n}(\nu)dP_{f,n}(\nu_1).$$

\begin{lemma}\label{Mix}\,\,
  The family  of distributions  (i) $(P_{f,n},\,n\in\N)$ (ii)
$(\tilde{P}_{f,n},\,n\in\N)$ are exponentially  tight.

\end{lemma}
\begin{proof}

(i)  As  this  family distributions  obey  a  large  deviation upper
bound  with  a  good  rate  function  $K_{\nu}(\omega),$ the family
$(P_{f,n},\,n\in\N)$ is exponentially tight. See, e.g. (Dembo  and
Zeitouni, 1998, Exercise~4.1.10(c)).

(ii) By  (i)  for  every  $\theta_2$  we  can  find  $K_{\theta_2},$
compact subset of $\skrid_{\skrim(\skrin\times\skrix)}$  such that,
we have
$$\limsup_{n\to \infty}\frac{1}{n}\log
P_{f,n}(K_{\theta_2}^c)\le -\theta_2.$$  Also  by Lemma~\ref{PA5},
for every $\theta_1$  we can  find $K_{\theta_1},$ compact subset of
$\skrim(\skrin\times\skrix)$ such that,  we have
$$\limsup_{n\to \infty}\frac{1}{n}\log P_{f,n}(K_{\theta_1}^c)\le
-\theta_1.$$

Take  $\theta=\min(\theta_1,\theta_2)$ and define the relatively
compact set $\Gamma_\theta$  by

$$\Gamma_{\theta}:=\Big\{(\nu_1,\nu)\in\skrim(\skrin\times\skrix)\times\skrid_{\skrim(\skrin\times\skrix)}:\,\nu_1\in
K_{\theta_1} \mbox{ and }\nu\in K_{\theta_2}\Big\}.$$  Now,  let
$\delta>0$    and notice  that,  for  sufficiently  large $n$  we
have that
$$\tilde{P}_{f,n}(\Gamma_{\theta}^c)\le \P\big\{ M_X\in K_{\theta_1}^c \big\}+ \P\big\{ L^X\in K_{\theta_2}^c
\big\}\le C(\theta)e^{-n(\theta-\delta)}.$$

Taking  limit  $n\to\infty$ followed  by   $\delta\downarrow 0$  of
above  inequality,  yields
$$\limsup_{n\to\infty}\frac{1}{n}\log\tilde{P}_{f,n}(\Gamma_{\theta}^c)\le
-\theta$$ which  proves the second  part  of  the  Lemma.
\end{proof}

Now, as  $J(\nu_1, \nu)$  is  lower  semi-continuous  by  the
continuity of the  relative  entropies,  and  by  Lemma~\ref{Mix}
the  families of distributions (i) $(P_{f,n},\,n\in\N)$ (ii)
$(\tilde{P}_{f,n},\,n\in\N)$  are  exponentially  tight,  we have
that the latter obeys  a  large  deviation principle with  good rate
function give  by  $J(\nu_1).$  (See, Biggins, 2004, Theorem~5(a)).

\subsection{Proof  of  Theorem~\ref{PA3a}}
We note  that in  case  of  this  theorem  $\gamma_t=\gamma,$
$\beta_t=\beta,$  and  hence $c_t=c$  for  all  $t\in(0,1].$
Therefore, Theorem~\ref{PA30} and  the  contraction  principle, (see
Dembo and Zeitouni, 1998, Theorem~4.2.1) imply  the  large deviation
principle  for  $M_X$ in  the  space $\skrim(\skrin\times\skrix)$
with good rate function
$$\begin{aligned}
\inf_{\nu\in\skrid_{\skrim}}\Big\{\tilde{J}(\omega,\nu):\omega=\nu_1\Big\}&=\inf_{\nu\in\skrid_{\skrim}}\Big\{H(\omega_{2,1}\,\|\,\mu)
+\sum_{a\in\skrix}\omega_{2}(a)\int_{0}^{1}H\Big(\omega(\cdot|a)\,\|\,\,\frac{c}{{f}}\otimes\nu_t(\cdot|a)\Big)dt:\omega=\nu_1\Big\}\\
&\ge\inf_{\nu\in\skrid_{\skrim}}\Big\{ H(\omega_{2,1}\,\|\,\mu)
+\sum_{a\in\skrix}\omega_{2}(a)H\Big(\omega(\cdot|a)\,\|\,\,\frac{c}{{f}}\otimes\int_{0}^{1}\nu_t(\cdot|a)dt\Big):\omega=\nu_1\Big\}\\
&=H(\omega_{2,1}\,\|\,\mu)+\sum_{a\in\skrix}\omega_{2}(a)H\Big(\omega(\cdot|a)\,\|\,\,\frac{c}{{f}}\otimes\hat\omega(\cdot|a)\Big)=J(\omega)
\end{aligned}$$

where    in  the  third  step,  we  have used  the  inequality
$$\nu_t(k|a)\le
\int_{0}^{1}\nu_t(k|a)dt=\int_{0}^{1}\nu_t^{\nu}(k|a)dt=\1-\sum_{i=0}^{k}\nu_1(i|a)=\1-\sum_{i=0}^{k}\omega(i|a)=\hat\omega(k|a)$$
for all $(k,a)\in\skrin\times\skrix$ and for all $t\in[0,1].$ This
 ends  the  proof  this  Theorem.
\subsection{Proof  of  Theorem~\ref{PA3}}

In the case of an preferential attachment graph, the function
$c=\gamma(a)+\beta(a)$ degenerates to a constant~$c=\gamma+\beta$
and $M_X=\skril\in\skrim(\skrin)$. Theorem~\ref{PA3a} and the
contraction principle imply a large deviation principle for $\skril$
with good rate function
$$J(\ell)=H\Big(\ell\,\|\,\sfrac{(\gamma+\beta)}{f}\otimes\hat{\ell}\Big)=I(\ell),$$
where
$\sfrac{(\gamma+\beta)}{f}\otimes\hat{\ell}(k)=\frac{(\gamma+\beta)}{f(k)}\hat{\ell}(k)$
and $\hat{\ell}(k)=\1-\sum_{j=0}^{k}\ell(k).$

\textbf{References}

\hangafter=1

\setlength{\hangindent}{2em}

 {\sc  Barab´asi, A.} and
{\sc  Albert, R.}(1999).
\newblock {Emergence of Scaling in Random Networks.}
\newblock{Science 286,509-512.}
\smallskip

\hangafter=1

\setlength{\hangindent}{2em}

{\sc  Biggins, J.D.}(2004).
\newblock{Large deviations for mixtures.}
\newblock{\emph{ El. Comm. Probab.9 60 71 (2004).}}
\smallskip

\hangafter=1

\setlength{\hangindent}{2em}

 {\sc Collevecchio,A.}, {\sc Cotar, C.} and
{\sc LiCalzi, M.} (2013).
\newblock{ On a preferential attachment
and generalized Polya's urn model. Ann. Appl. Probab. 23,
1219–1253.}

\hangafter=1

\setlength{\hangindent}{2em}

{\sc Choi, J.} and {\sc Sethuraman, S.}(2011)
\newblock{Large deviations of  the degree structures in
P.A schemes. }
\newblock \emph{The  annals  of  applied  probability, 23, 722-763.}
\smallskip

\hangafter=1 \setlength{\hangindent}{2em} {\sc Dereich, S.} and {\sc
Morters, P.}(2009).
\newblock{Random networks with sublinear preferential  attachement: Degree evolutions.}
\newblock{\emph{Electronic Journal of Probability, 14, pp.
1222-1267.}}
\smallskip

\hangafter=1 \setlength{\hangindent}{2em}

{\sc Doku-Amponsah, K.}(2006)
\newblock{Large deviations and basic information theory for hierarchical and networked data structures.}
\newblock PhD Thesis, Bath (2006).
\smallskip

\hangafter=1 \setlength{\hangindent}{2em} {\sc Bryc, W.}, {\sc
~Minda, D.} and {\sc ~Sethuraman, S.}(2009).
\newblock{Large deviations for the leaves in some random trees .}
\newblock { \emph{Adv. in  Appl. Probab. Volume 41, Number 3 (2009),
845-873.}http://dx.doi.org/10.1239/aap/1253281066}
\smallskip

\hangafter=1 \setlength{\hangindent}{2em}

 {\sc Doku-Amponsah, K.}  and
{\sc~M\"orters, P.}(2010).
\newblock{Large deviation principle for  empirical measures of
coloured random graphs.}
\newblock {\emph{The  annals  of  Applied Probability}}, 20,
1989-2021(2010).http://dx.doi.org/10.1214/09-AAP647

\hangafter=1

\setlength{\hangindent}{2em}

 {\sc Dembo, A.} and {\sc
O.~Zeitouni, O.}(1998).
\newblock Large deviations techniques and applications.
\newblock Springer, New York, (1998).

\hangafter=1

\setlength{\hangindent}{2em}

{\sc  Krapivsky, P. L.},{ Redner, S.} and { Leyvraz, F.}  (2000).
Connectivity of growing random networks. Physical review letters 85,
4629.

\hangafter=1
\setlength{\hangindent}{2em}

{\sc M. Newman,M}, {\sc  Barab´asi, A.-L.} and {\sc Watts, D. J.}
(2006).
\newblock The structure and dynamics of networks. Princeton
University Press.

\hangafter=1

\setlength{\hangindent}{2em}

{\sc Newman, M. E. J.} (2003). The structure and function of complex
networks. SIAM review.

\hangafter=1 \setlength{\hangindent}{2em}

 {\sc Lawrence, S.} and
{~Giles, C.L.}(1998)(1999).
\newblock {Science 280, 98 (1998); Nature 400, 107 (1999).}
\smallskip

\hangafter=1 \setlength{\hangindent}{2em}

 {\sc Rudas, B.},{\sc Toth,
B. } and {\sc ~Valko, B.}(2008).
\newblock{Random Trees and General Branching Processes.}
\newblock {http://arxiv.org/abs/math/0503728}
\smallskip

\bigskip

\end{document}